\documentclass[10pt,reqno,a4paper]{amsart}
\usepackage{euscript}
\usepackage{amssymb}
\usepackage{breqn}
\usepackage{graphicx}
\usepackage{color}
\usepackage{overpic}
\usepackage[T1]{fontenc} 
\usepackage{inputenc}
\usepackage{float}
\usepackage{hyperref}
\usepackage{amsthm,amssymb,amsmath,amsfonts,amscd}
\usepackage[table]{xcolor}
\usepackage{graphicx,epsfig,latexsym}
\usepackage{cite,calc,xcolor,dsfont,epstopdf}
\usepackage[lined,commentsnumbered]{algorithm2e}
\usepackage{tabularx}
\usepackage{caption}

\newcommand{\ie}{{\em i.e.,} }

\newtheorem{thm}{Theorem}[section]
\newtheorem{cor}[thm]{Corollary}
\newtheorem{lem}[thm]{Lemma}
\newtheorem{prop}[thm]{Proposition}
\theoremstyle{definition}
\newtheorem{defn}[thm]{Definition}

\newtheorem{rem}[thm]{Remark}

\numberwithin{equation}{section}
\usepackage{epstopdf}

\def\Blem {\begin{lem}}
\def\Elem {\end{lem}}
\def\be {\begin{equation}}
\def\ee {\end{equation}}
\def\ba {\begin{eqnarray}}
\def\ea {\end{eqnarray}}
\def\bes {\begin{equation*}}
\def\ees {\end{equation*}}
\def\bas {\begin{eqnarray*}}
\def\eas {\end{eqnarray*}}
\def\bpr {\begin{proof}}
\def\epr {\end{proof}}

\newtheorem{mtheorem}{Theorem}

\begin{document}
\allowdisplaybreaks
\title[Global Centers and Phase Portraits in Generalized Duffing Oscillators]{Global Centers and Phase Portraits in Generalized Duffing Oscillators: A Comprehensive Study of the Center-Focus Problem
}

\begin{abstract}
This work presents a comprehensive study of the generalized Duffing oscillator, a fundamental model in nonlinear dynamics described by the system
\[
\dot{x} = y, \quad \dot{y} = -\alpha y - \epsilon x^m - \sigma x,
\]
where \(\epsilon \neq 0\) and \(m \geq 1\). We focus on the topological classification of phase portraits, the characterization of global centers, and the  absence of limit cycles for $\alpha\neq0$. For the linear case (\(m = 1\)), we establish necessary and sufficient conditions for the origin to be a global center, showing that this occurs if, and only if, \(\alpha = 0\) and \(\epsilon + \sigma > 0\). For the nonlinear case (\(m > 1\)), we prove that the origin is a global center if, and only if, \(m\) is odd, \(\sigma, \epsilon > 0\), \(\alpha = 0\). Additionally, we classify the global phase portraits for every \(m\), demonstrating the rich dynamical behavior of the system and detect homoclinic, heteroclinic and double-homoclinic cycle for \(\alpha=0.\) Using the Bendixson-Dulac criterion, we rule out the existence of limit cycles for $\alpha\neq 0$, further clarifying the behavior of the system. Our results resolve the center-focus problem for the degenerate case \(\alpha = 0\) and provide a complete characterization of global centers for generalized Duffing oscillators of odd degrees. These findings contribute to the broader understanding of nonlinear dynamical systems and have potential applications in modeling oscillatory phenomena. 
\end{abstract}

\author{Gabriel Rond\'on \({}^{1}\)  and Nasrin Sadri\({}^{2}\)}

\address{$^{1}$ Departament de Matemàtiques, Edifici Cc, Universitat Autònoma de Barcelona, 08193 Bellaterra, Barcelona, Catalonia, Spain}
\email{garv202020@gmail.com}

\address{$^2$ School of Mathematics, Statistics and Computer Science, College of Science, University of Tehran, Tehran, Iran and School of Mathematics, Institute for Research in Fundamental Sciences (IPM)}
\email{ n.sadri@ut.ac.ir}



\maketitle

\section{Introduction}
The Duffing oscillator and its generalizations are fundamental models in the study of nonlinear dynamical systems, with applications spanning mathematics, physics, and biology. These systems are described by the following set of polynomial ordinary differential equations:
\begin{equation}\label{gdo}
    \dot{x} = y, \quad \dot{y} = -\alpha y - \epsilon x^m - \sigma x,
\end{equation}
where \(\epsilon \neq 0\) and \(m \geq 1\). These equations belong to the Li\'enard family of differential systems \cite{LLIBRE202266} and have been widely used to model phenomena such as the interaction between diffusion and nonlinear reaction terms in the Newell-Whitehead-Segel equation. In particular, the traveling wave solutions of the Newell-Whitehead-Segel equation can be represented by generalized Duffing systems.

Several studies have explored the integrability of the force-free Duffing-van der Pol oscillator under specific parametric conditions. Chandrasekar and et al. (2004) demonstrated its integrability for certain parameter restrictions, and later, in 2006, they extended their findings by identifying additional integrable cases for arbitrary values of the exponent \cite{Chandrasekar, Chandrasekar1}.
The exact solutions of the Duffing oscillator equation have been examined in earlier works, including those by Parthasarathy and Lakshmanan (1990) \cite{Parthasarathy}, as well as Salas and Castillo (2014) \cite{Salas}.
Further investigations into the integrability of force-free oscillators have been conducted in subsequent years. Demina (2018) systematically analyzed conditions under which the classical force-free Duffing and Duffing-van der Pol oscillators admit Liouvillian first integrals \cite{Demina}. Late, in 2021, she
studied the liouvillian integrability for the classical force-free generalized Duffing oscillators \cite{Demina1}. 
Similarly, Ruiz and Muriel (2018) explored the integrability of generalized force-free Duffing–van der Pol equations through the framework of \(\lambda\)-symmetries and solvable structures \cite{Ruiz}. Additionally, Stachowiak (2019) investigated the existence of first integrals for both Duffing and van der Pol oscillators \cite{Stachowiak}.

The case \( m > 1 \) corresponds to a nonlinear system, where the term \( \epsilon x^m \) introduces nonlinearities essential for capturing complex behaviors such as bifurcations, chaos, and multiple equilibria. On the other hand, the case \( m = 1 \) reduces the system to a damped linear oscillator, which describes physical phenomena such as harmonic oscillations in mechanical, electrical, or biological systems. This linear limit not only provides a foundation for understanding the behavior of the system in small-amplitude regimes but is also relevant in contexts where nonlinear interactions are negligible.

In this work, we conduct a comprehensive study of the generalized Duffing oscillator \eqref{gdo}, focusing on the topological classification of phase portraits, the characterization of global centers, the analysis of finite and infinite equilibrium points and the absence of limit cycles for $\alpha\neq 0$. Our first main result is presented in the following theorem.

\begin{mtheorem}\label{1}
    The following statements hold for system \eqref{gdo}:
    \begin{itemize}
        \item[(a)] 
        When \(\alpha = 0,\) the global phase portrait is topologically equivalent to one of the phase portraits in Figure \ref{M44}.
        \item[(b)] 
        For \(m = 1\), the global phase portrait is topologically equivalent to one of the phase portraits in Figure \ref{M11}.
        \item[(c)]
        If \(m \) is even, the global phase portrait is topologically equivalent to one of the phase portraits in Figure \ref{M22}.
        \item[(d)] 
       If \(m>1\) is odd, the global phase portrait is topologically equivalent to one of the phase portraits in Figure \ref{M33}.
    \end{itemize}
\end{mtheorem}
 
The dynamics of the system depend on several parameters: \(\alpha\), \(\epsilon\), \(\sigma\), and \(m\). The parameter \(\alpha\) plays a crucial role in determining the nature of the equilibrium points, particularly in the transition between dissipative and conservative dynamics. When \(\alpha \neq 0\), the system exhibits well-defined phase portraits, as described in Theorem \ref{1}. However, the case \(\alpha = 0\) presents a critical scenario where the eigenvalues of the equilibrium points are purely imaginary. This degenerate case lies at the boundary between distinct dynamical behaviors, making it a particularly intriguing and challenging case to analyze.

When the linearization of a system yields purely imaginary eigenvalues, the nature of the equilibrium point in the full nonlinear system is not determined solely by the linear terms. The nonlinear terms play a decisive role in distinguishing whether the equilibrium is a \textit{center} (surrounded by closed orbits) or a \textit{focus} (with spiraling trajectories). This is known as the \textit{center-focus problem}, a classical issue in the theory of dynamical systems that has attracted significant attention due to its mathematical subtlety and practical implications. Resolving this problem is essential for a complete understanding of the behavior of the system, particularly in the transition regime as \(\alpha\) approaches zero.

To address this problem, we first in the following theorem establish necessary and sufficient conditions under which the origin of system \eqref{gdo} is a center, highlighting the role of the parameters \(\alpha\), \(\epsilon\), \(\sigma\), and the degree \(m\).

The classification of centers depends on whether the linear part of the system
\begin{equation}\label{red1}
\dot{x} = X(x, y), \quad \dot{y} = Y(x, y),
\end{equation}
vanishes at the origin. Based on this, there are three distinct types of centers. After applying appropriate variable transformations and rescaling, system \eqref{red1} with a center can be expressed in one of the following canonical forms:

\begin{equation*}
\begin{aligned}
&\dot{x} = y + P(x, y), && \dot{y} = -x + Q(x, y), \\
&\dot{x} = y + P(x, y), && \dot{y} = Q(x, y), \\
&\dot{x} = P(x, y),     && \dot{y} = Q(x, y).
\end{aligned}
\end{equation*}
Here, \( P(x, y) \) and \( Q(x, y) \) are polynomials satisfying \( P(0, 0) = Q(0, 0) = 0 \), with all terms being of degree two or higher. Depending on the form, the origin is referred to as a linear type center, a nilpotent center, or a degenerate center, respectively, \cite{chen1,chen2}.

\begin{mtheorem}\label{prop1}
    The origin is the unique equilibrium of the polynomial differential system \eqref{gdo} and is of linear center type if, and only if:
    \begin{itemize}
        \item \( m = 1 \), \(\epsilon + \sigma > 0\) and \(\alpha = 0\); or
        \item \( m > 1 \) is odd, \(\sigma, \epsilon > 0\), and \(\alpha = 0\).
    \end{itemize}
\end{mtheorem}
This result is a crucial step in our analysis, as it identifies the specific parameter regime in which the system exhibits a center at the origin. It also sets the stage for the study of \textit{global centers}, where the entire phase space is filled with closed orbits. For an oscillator, having a global center means that the system will oscillate periodically for any initial condition, and its motion will always remain bounded, closed, and non-decaying. There is no damping, like friction, and no energy input or loss, so the oscillator moves in perfect cycles forever. Physically, this represents an idealized oscillator that conserves energy completely.

For the linear case (\( m = 1 \)), if the origin is the unique equilibrium point (both finite and at infinity) and is a center (i.e., \(\alpha = 0\) and \(\epsilon + \sigma > 0\)), it automatically becomes a global center. This is because the absence of other equilibria ensures that all trajectories in the phase space are closed orbits surrounding the origin. Thus, no further analysis is required to establish the global center property in this case.

However, for the nonlinear case (\( m > 1 \)), the situation is more intricate. In addition to addressing the center-focus problem, we aim to establish necessary and sufficient conditions for the existence of a global center in the system when \(\alpha = 0\). A global center is characterized by the property that all trajectories in the phase space are closed orbits, implying a conservative structure without dissipation or amplification. This is a stronger and more restrictive condition than the existence of a local center, and its analysis requires a deeper exploration of the invariants and symmetries of the system. The next result provides a characterization of the generalized Duffing oscillators \eqref{gdo} with a global center at the origin, offering new insights into the global dynamics of the system.

\begin{mtheorem}\label{teoB}
    The polynomial generalized Duffing oscillator \eqref{gdo} has a global center at the origin if, and only if, the following conditions are satisfied:
        \begin{itemize}
            \item For \( m = 1 \): \(\epsilon+\sigma > 0 \) and \( \alpha = 0 \); or
            \item For \( m > 1 \): \( m \) is odd, \( \sigma, \epsilon > 0 \) and \( \alpha = 0 \).
        \end{itemize}
\end{mtheorem}

In~\cite{ref1}, the authors proved that polynomial differential systems of even degree cannot have global centers, as such systems necessarily possess trajectories that escape to infinity in both forward and backward time. This result is summarized in Theorem~\ref{evendeg}, and it implies that, in order for the system under study to possibly exhibit a global center, the degree \( m \) must necessarily be odd.

The paper is organized as follows. Section \ref{sec:Preliminaries} is divided into six subsections. In subsection \ref{sec:center}, we analyze the conditions under which the origin of the polynomial differential system is a center, focusing on the existence of first integrals and applying the Poincar\'e-Lyapunov Theorem. Subsection \ref{sec:compac} reviews the Poincar\'e compactification and presents the key formulas required for the subsequent analysis. Subsection \ref{sec:globalc} focuses on the characterization of global centers, providing necessary and sufficient conditions for their existence. Subsection \ref{sec:vertical} discusses the use of quasi-homogeneous blow-ups and Newton diagram to determine the local phase portraits of equilibrium points with a zero linear structure. In subsection \ref{sec:ben_crit} we recall the Bendixson-Dulac criterion, which provides conditions to rule out the existence of limit cycles in planar differential systems based on the divergence of the vector field. Finally, subsection \ref{sec:newtonian} presents the Newtonian system and discusses its critical points.

In section \ref{main_A}, we prove Theorem \ref{1}, which classifies the global phase portraits of the system. Section \ref{main_BC} is dedicated to proving Theorems \ref{prop1} and \ref{teoB}, which establish conditions for the origin to be a center and characterize the existence of global centers, respectively. 

\section{Preliminaries}\label{sec:Preliminaries}

This section presents foundational results that will be essential for proving Theorems \ref{1}, \ref{prop1} and \ref{teoB}.

\subsection{The Center Conditions}\label{sec:center}

Consider the polynomial differential system given by
\begin{equation}\label{eq_center}
\dot{x}= -y + P_n(x, y),\quad \dot{y}=x + Q_n(x, y)),
\end{equation}
where \( P_n \) and \( Q_n \) are polynomials of degree \( n \), which do not have neither constant nor linear terms. The Jacobian matrix of the system at the origin has purely imaginary eigenvalues, meaning the origin is either a focus or a center. 

Recall that a point $p$ is called a \textit{center} of a differential system in $\mathbb{R}^2$ if there exists a neighbourhood $U$ around $p$ such that $U \setminus \{p\}$ is entirely filled with periodic orbits. The largest connected set of periodic orbits that encircles the center $p$ and has $p$ on its boundary is known as the \textit{period annulus} of the center $p$. Additionally, $p$ is termed a \textit{global center} if its period annulus is $\mathbb{R}^2 \setminus \{p\}$, see \cite{chen3}.

A non-constant analytic function $H: \Omega \subset \mathbb{R}^2 \to \mathbb{R}$, defined in a neighborhood $\Omega$ of the origin, is called a \textit{first integral} of system \eqref{eq_center} if it remains constant along any solution curve $\gamma$, which is equivalent to
\begin{equation*}
\frac{\partial H}{\partial x} \dot{x} + \frac{\partial H}{\partial y} \dot{y} \bigg|_\gamma \equiv 0. \quad \quad
\end{equation*}
To determine whether the origin is a center, we employ the \textit{Poincaré-Lyapunov Theorem}. For more details, see, for instance, \cite{ref1, ref2, ref3, ref4}:

\begin{thm}[Poincar\'e-Lyapunov]\label{PL}
The polynomial differential system \eqref{eq_center} has a center at the origin if, and only if, it admits a local analytic first integral of the form
\begin{equation*}
H(x, y) = x^2 + y^2 + \sum_{p=3}^\infty H_p(x, y),
\end{equation*}
where $H_p(x, y)$ is a homogeneous polynomial of degree $p$, given by
\begin{equation*}
H_p(x, y) = \sum_{\ell=0}^p q_{p-\ell, \ell} x^{p-\ell} y^\ell.
\end{equation*}
Moreover, the existence of a formal first integral $H$ of the above form implies the existence of a local analytic first integral.
\end{thm}

\subsection{Poincar\'e compactification for vector fields in the plane}\label{sec:compac}

To analyze the global dynamics of a planar polynomial differential system \(X = (P, Q)\), it is essential to classify the local phase portraits of its finite and infinite equilibrium points within the Poincar\'e disc.

Let \(\mathbb{S}^2 = \{\mathbf{z} \in \mathbb{R}^3 : ||\mathbf{z}|| = 1\}\) be the unit sphere in \(\mathbb{R}^3\). The polynomial vector field \(X\) induces an analytic vector field on \(\mathbb{S}^2\), denoted as \(p(X)\) (see, for example, \cite[Chapter 5]{MR2256001} or \cite{10.2307/2001320}).

The vector field \(p(X)\) facilitates the examination of the dynamics of \(X\) in the vicinity of infinity, specifically near the equator \(\mathbb{S}^1 = \{\mathbf{z} \in \mathbb{S}^2 : z_3 = 0\}\).

We treat the sphere as a smooth manifold to derive the analytical expression for \(p(X)\). We utilize six local charts defined as \(U_i = \{\mathbf{z} \in \mathbb{S}^2 : z_i > 0\}\) and \(V_i = \{\mathbf{z} \in \mathbb{S}^2 : z_i < 0\}\) for \(i = 1, 2, 3\). The corresponding coordinate maps \(\varphi_i : U_i \rightarrow \mathbb{R}^2\) and \(\psi_i : V_i \rightarrow \mathbb{R}^2\) are given by \(\varphi_k(\mathbf{z}) = \psi_k(\mathbf{z}) = \left(\frac{z_m}{z_k}, \frac{z_n}{z_k}\right)\) for \(m < n\) and \(m, n \ne k\).

Let \((u, v)\) be the local coordinates on \(U_i\) and \(V_i\) for \(i = 1, 2, 3\). According to \cite[Chapter 5]{MR2256001}, the vector field \(p(X)\) in these local charts can be expressed as:
\begin{equation}\label{poincare_comp}
\begin{array}{rl}
(\dot{u}, \dot{v}) =& \left(v^n \left(-u P\left(\frac{1}{v}, \frac{u}{v}\right) + Q\left(\frac{1}{v}, \frac{u}{v}\right)\right), -v^{n+1} P\left(\frac{1}{v}, \frac{u}{v}\right)\right) \quad \text{in } U_1; \vspace{0.3cm} \\
(\dot{u}, \dot{v}) =& \left(v^n \left(-u Q\left(\frac{u}{v}, \frac{1}{v}\right) + P\left(\frac{u}{v}, \frac{1}{v}\right)\right), -v^{n+1} Q\left(\frac{u}{v}, \frac{1}{v}\right)\right) \quad \text{in } U_2; \vspace{0.3cm} \\
(\dot{u}, \dot{v}) =& \left(P(u,v), Q(u,v)\right) \quad \text{in } U_3,
\end{array}
\end{equation}
where \(n\) denotes the degree of the polynomial vector field \(X\). We recall that the expressions for the vector field \(p(X)\) in the local chart \((V_i, \psi_i)\) are identical to those in the local chart \((U_i, \varphi_i)\) multiplied by \((-1)^{n-1}\) for \(i = 1, 2, 3\).

The points at infinity in all local charts take the form \((u, 0)\). The infinity \(\mathbb{S}^1\) remains invariant under the flow of \(p(X)\).

The equilibrium points of the vector field \(X\) are referred to as the {\it finite} equilibrium points of \(X\) or of \(p(X)\), while the equilibrium points of the vector field \(p(X)\) in \(\mathbb{S}^1\) are termed the {\it infinite} equilibrium points of \(X\) or of \(p(X)\).

\subsection{Characterization of global centers}\label{sec:globalc}

The following result provides the necessary and sufficient conditions for a planar polynomial differential system to exhibit a global center. The proof of this result can be found in \cite{doi:10.1080/14689367.2023.2228737}.

\begin{prop}\label{prop_main}
A polynomial differential system of degree \(n\) in \(\mathbb{R}^2\) that does not possess a line of equilibrium points at infinity has a global center if, and only if, it has a unique finite equilibrium point that is a center, and all the local phase portraits of the infinite equilibrium points (if they exist) are comprised of two hyperbolic sectors, with their two separatrices located on the infinite circle.
\end{prop}

In \cite{GalVill}, the authors proved that polynomial differential systems of even degree cannot have global centers, as such systems necessarily possess orbits that escape to infinity in both forward and backward time. This result is summarized in the following theorem.

\begin{thm}\label{evendeg}
Consider a polynomial differential system of the form $\dot{x} = P(x, y)$, $\dot{y} = Q(x, y)$. If the system has an even degree, then it cannot admit a global center.
\end{thm}

\subsection{Quasi-homogeneous blow-up}\label{sec:vertical}

Assume that the origin is a singular point of a smooth vector field \( X \) defined on \( \mathbb{R}^2 \). We consider the transformation:
\begin{equation*}
\phi : \mathbb{S}^1 \times \mathbb{R} \to \mathbb{R}^2, \quad 
(\theta, \rho) \mapsto \left( \rho^{\alpha} \cos \theta,\, \rho^{\beta} \sin \theta \right),
\end{equation*}
where \( \alpha, \beta \in \mathbb{N} \), with \( \alpha, \beta \geq 1 \), are appropriately chosen. In the particular case \( (\alpha, \beta) = (1,1) \), this corresponds to the standard homogeneous case. This mapping allows us to lift \( X \) to a smooth vector field \( \tilde{X} \) on \( \mathbb{S}^1 \times \mathbb{R} \) via \( \phi_*(\tilde{X}) = X \). 

To regularize the behavior near the origin, we divide \( \tilde{X} \) by \( \rho^k \) for some \( k \geq 1 \), yielding a new smooth field \( \overline{X} = \dfrac{1}{\rho^k} \tilde{X} \), which remains as regular as possible along the circle \( \mathbb{S}^1 \times \{0\} \).

Directional blow-ups are then introduced as follows:
\begin{align*}
\text{Positive } x\text{-axis:} \quad & (\bar{x}, \bar{y}) \mapsto (\bar{x}^{\alpha}, \bar{x}^{\beta} \bar{y}) \quad \quad \Rightarrow \hat{X}^{x}_{+}, \\
\text{Negative } x\text{-axis:} \quad & (\bar{x}, \bar{y}) \mapsto (-\bar{x}^{\alpha}, \bar{x}^{\beta} \bar{y}) \,\quad \Rightarrow \hat{X}^{x}_{-}, \\
\text{Positive } y\text{-axis:} \quad & (\bar{x}, \bar{y}) \mapsto (\bar{x} \bar{y}^{\alpha}, \bar{y}^{\beta}) \quad \quad\Rightarrow \hat{X}^{y}_{+}, \\
\text{Negative } y\text{-axis:} \quad & (\bar{x}, \bar{y}) \mapsto (\bar{x} \bar{y}^{\alpha}, -\bar{y}^{\beta}) \,\quad \Rightarrow \hat{X}^{y}_{-}.
\end{align*}

These blow-ups also induce rescaled systems \( \overline{X}^{x}_{\pm} \) and \( \overline{X}^{y}_{\pm} \), obtained by dividing the corresponding fields by appropriate powers of \( \bar{x} \) or \( \bar{y} \), depending on the direction.

When \( \alpha \) (respectively \( \beta \)) is odd, the analysis along the positive \( x \)-axis (respectively \( y \)-axis) reflects the behavior on the negative side as well.

One natural question is how to determine suitable values for the pair \((\alpha, \beta)\) when performing a quasi-homogeneous blow-up. A useful tool for this purpose is the \textit{Newton diagram}, which we now proceed to define.

Let
\begin{equation*}
X = P(x, y) \frac{\partial}{\partial x} + Q(x, y) \frac{\partial}{\partial y}
\end{equation*}
be a polynomial vector field with an isolated critical point at the origin. Assume that
\begin{equation*}
P(x, y) = \sum_{i + j \geq 1} p_{ij} x^i y^j, \qquad Q(x, y) = \sum_{i + j \geq 1} q_{ij} x^i y^j.
\end{equation*}
We define the \textit{support} of \( X \) as the subset
\begin{equation*}
S = \{ (i - 1, j) \mid p_{ij} \neq 0 \} \cup \{ (i, j - 1) \mid q_{ij} \neq 0 \} \subset \mathbb{R}^2.
\end{equation*}
The \textit{Newton polygon} associated with \( X \) is given by the convex hull \( \Gamma \) of the union:
\begin{equation*}
P = \bigcup_{(r,s) \in S} \{ (\bar{r}, \bar{s}) \in \mathbb{R}^2 \mid \bar{r} \geq r,\, \bar{s} \geq s \}.
\end{equation*}
The \textit{Newton diagram} associated with \( X \) consists of the union \( \gamma \) of the compact faces \( \gamma_k \) forming the Newton polygon \( \Gamma \), which we enumerate from the left to the right.

For further details on quasi-homogeneous blow-up and Newton diagram, refer to \cite[section 3.3]{MR2256001}.

\subsection{Bendixson-Dulac Criterion}\label{sec:ben_crit}
 In what follows, we present a fundamental result that allows us to establish the non-existence of closed orbits for a \(\mathbb{C}^1\) planar differential system of the form
\begin{equation}\label{Ben_C}
  \dot{x} = P(x, y), \qquad \dot{y} = Q(x, y),
\end{equation}
where \(P\) and \(Q\) are continuously differentiable functions defined on \(\mathbb{R}^2\).

The \textit{Bendixson criterion} (see \cite{BC1} and \cite{MR2256001}) states the following: Let \(D\) be a simply connected domain in \(\mathbb{R}^2\). If the divergence of the system,
\begin{equation*}
\frac{\partial P}{\partial x} + \frac{\partial Q}{\partial y},
\end{equation*}
has a constant sign in \(D\) (i.e., it is either always non-positive or always non-negative), except possibly on a subset of \(D\) with zero Lebesgue measure, then the differential system \eqref{Ben_C} has no closed orbits entirely contained in \(D\).

This criterion was later generalized by \textit{H. Dulac} (see \cite{MR2256001}) as follows: Let \(D\) be a simply connected domain in \(\mathbb{R}^2\), and let \(f(x, y)\) be a \(\mathbb{C}^1\) function defined on \(D\). If the expression
\begin{equation*}
\frac{\partial (fP)}{\partial x} + \frac{\partial (fQ)}{\partial y}
\end{equation*}
has a constant sign in \(D\), except possibly on a subset of \(D\) with zero Lebesgue measure, then the differential system \eqref{Ben_C} has no closed orbits entirely contained in \(D\).

\subsection{Newtonian system}\label{sec:newtonian} Let us start by introducing the Hamiltonian systems.
\begin{defn}
Let \(E\) be an open subset of \(\mathbb{R}^{2},\) and consider a function \(H \in C^2(E)\) where \(H = H(x,y)\) with \(x, y \in \mathbb{R}.\) A system of differential equations of the form
\begin{eqnarray*}
\dot{x} = \frac{\partial H}{\partial y}, \qquad \dot{y} = -\frac{\partial H}{\partial x}, 
\end{eqnarray*}
is known as a Hamiltonian system \(1\) degree of freedom on \(E.\)
\end{defn}
Hamiltonian systems are inherently conservative, meaning that the Hamiltonian function (or total energy) \(H(x, y)\) remains constant along the trajectories of the system.

The \textit{Newtonian system} is a specific example of a Hamiltonian system with one degree of freedom which is  given by $\ddot{x} = f(x), $ where \(f \in C^1(a,b).\) This second-order differential equation can be reformulated as a first-order system in \(\mathbb{R}^2\)
\begin{eqnarray}\label{ns}
 \dot{x} = y, \qquad \dot{y} = f(x). 
\end{eqnarray}
The total energy for this system is given by \( H(x,y) = T(y) + U(x) \), where \( T(y) = \frac{y^2}{2} \) represents the kinetic energy, and \( U(x) = - \int_{x_0}^{x} f(s) ds \)
defines the potential energy. With this formulation of \(H(x,y),\) the Newtonian system \eqref{ns} can be expressed as a Hamiltonian system.
The following theorem describes the equilibrium points of the Newtonian system.
\begin{thm}\cite[Section 2.14]{perko}\label{maxmin}
The critical points of the Newtonian system \eqref{ns} are located on the \(x\)-axis. 
A point \((x_0,0)\) is a critical point of \eqref{ns} if, and only if, it is a critical point of the function \(U(x),\) meaning it is a root of \(f(x).\) Furthermore:
\begin{itemize} 
\item If \((x_0,0)\) is a strict local maximum of the analytic function \(U(x),\) it corresponds to a saddle point of \eqref{ns}.
\item If \((x_0,0)\) is a strict local minimum \(U(x),\) it represents a center of \eqref{ns}.
\item If \((x_0,0)\) is a horizontal inflection point of \(U(x),\) it forms a cusp in the phase portrait of \eqref{ns}.
\item The phase portrait of \eqref{ns} is symmetric with respect to the \(x\)-axis. 
\end{itemize}
\end{thm}

\section{Proof of Theorem \ref{1}}\label{main_A}
In this section, we focus on proving Theorem \ref{1}. Due to the length of the proof, we have divided it into two subsections. Section \ref{sec1} is dedicated to analyzing the local dynamics of Generalized Duffing Oscillators, while Section \ref{GB} examines the global dynamics of the system.

\subsection{Finite equilibrium points}\label{sec1}

In what follows, we analyze the equilibrium points of the Generalized Duffing Oscillator system \eqref{gdo}. The existence and stability of these equilibria play a crucial role in understanding the dynamics of the system. We begin by deriving the conditions under which the system has equilibrium points and discussing their stability for different values of the parameters \( \alpha \), \( \sigma \), and \( \epsilon \). The following proposition summarizes the main results regarding the finite equilibrium points of the system.

\begin{prop}\label{prop_sing} 
Consider system \eqref{gdo}. Then the following statements hold.
\begin{itemize}
    \item[(i)] 
    If $ m = 1 $, the generalized Duffing oscillator \eqref{gdo} has a unique equilibrium at the origin $(0, 0)$. Its stability is determined as follows:
    \begin{itemize}
        \item[$\bullet$] For $ \epsilon + \sigma < 0 $, the origin is a saddle point.
        \item[$\bullet$] For $ \epsilon + \sigma > 0 $:
        \begin{itemize}
            \item[$\triangleright$] When $ \alpha > 0 $ (resp. $ \alpha < 0 $), the origin is a stable (resp. unstable) equilibrium point. Furthermore, it is a focus (resp. node) if $ \alpha^2 < 4(\epsilon + \sigma) $ (resp. $ \alpha^2 > 4(\epsilon + \sigma) $).
            \item[$\triangleright$]  When $ \alpha = 0 $, the origin is a center.
        \end{itemize}
    \end{itemize}

    \item[(ii)] 
    If $ m > 1 $, the generalized Duffing oscillator \eqref{gdo} has at most three equilibria
    \begin{eqnarray}\label{eqneeed}
        (0, 0) \quad \text{and} \quad E^{\pm} = \left( \pm \sqrt[m-1]{-\frac{\sigma}{\epsilon}}, 0 \right).
    \end{eqnarray}
    Even more,
    \begin{itemize}
        \item[$\bullet$] if $ m $ is odd, $ E^{\pm} $ exist only when $ \sigma $ and $ \epsilon $ have opposite signs.
        \item[$\bullet$] if $ m $ is even, $E^{+}$ always exists.
    \end{itemize}

    \item[(iii)] 
    The stability of the equilibria for $ m > 1 $ is summarized in Table \eqref{tab1} for odd $ m $ and Table \eqref{tab2} for even $ m $.
\end{itemize}
\end{prop}

\begin{small}
\centering
\begin{tabular}{cc}
    \begin{minipage}{.5\linewidth}
    \begin{tabular}{|c|c|c|c|c|c|c|}
\hline
\(\sigma\) & \(\epsilon\) & \(\alpha\) &\((0,0)\)& \(E^+\) & \(E^-\)  \\
\hline
\(+\) & \(+\) & \(+\) & \cellcolor{blue}\mbox{stable}   &--- &--- \\
\hline
\(+\) & \(+\) & \(-\)  & \cellcolor{red}\mbox{unstable}   &--- &---\\
\hline
 \(-\) & \(-\)   & \(+\) & \cellcolor{green}\mbox{saddle} &--- &--- \\
\hline
\(-\) & \(-\)   & \(-\)  & \cellcolor{green}\mbox{saddle}  &--- &--- \\
\hline
\(+\) & \(-\) & \(+\) & \cellcolor{blue}\mbox{stable}    &\cellcolor{green}\mbox{saddle} &\cellcolor{green}\mbox{saddle}\\
\hline  
\(+\) & \(-\) & \(-\) &\cellcolor{red}\mbox{unstable}     &\cellcolor{green}\mbox{saddle} &\cellcolor{green}\mbox{saddle}\\
\hline  
\(-\) & \(+\) & \(+\) &\cellcolor{green}\mbox{saddle} &\cellcolor{blue}\mbox{stable} &\cellcolor{blue}\mbox{stable}\\
\hline
\(-\) & \(+\) & \(-\) &\cellcolor{green}\mbox{saddle} &\cellcolor{red}\mbox{unstable} &\cellcolor{red}\mbox{unstable} \\
\hline
\end{tabular}
\captionof{table}{ \label{tab1} Type of equilibria when \(m\) is odd.}
\end{minipage} &

\begin{minipage}{.5\linewidth}
\begin{tabular}{|c|c|c|c|c|c|c|}
\hline
\(\sigma\) & \(\epsilon\) & \(\alpha\) &\((0,0)\)& \(E^+\)   \\
\hline
\(+\) & \(+\) & \(+\) & \cellcolor{blue}\mbox{stable}    &--- \\
\hline
\(+\) & \(+\) & \(-\)  & \cellcolor{red}\mbox{unstable} &---  \\
\hline
\(-\) & \(-\)   & \(+\) & \cellcolor{green}\mbox{saddle} &---  \\
\hline
\(-\) & \(-\)   & \(-\)  & \cellcolor{green}\mbox{saddle} &---  \\
\hline
\(+\) & \(-\) & \(+\) & \cellcolor{blue}\mbox{stable}  &\cellcolor{green}\mbox{saddle} \\
\hline  
\(+\) & \(-\) & \(-\) &\cellcolor{red}\mbox{unstable}    &\cellcolor{green}\mbox{saddle} \\
\hline  
\(-\) & \(+\) & \(+\) &\cellcolor{green}\mbox{saddle}  &\cellcolor{blue}\cellcolor{blue}\mbox{stable}  \\
\hline
\(-\) & \(+\) & \(-\) &\cellcolor{green}\mbox{saddle} &\cellcolor{red}\mbox{unstable}  \\
\hline
\end{tabular}
\captionof{table}{\label{tab2} Type of equilibria when \(m\) is even.}
\end{minipage} 
\end{tabular}
\end{small}

\begin{proof}
We analyze the equilibria and their stability for the Generalized Duffing Oscillator \eqref{gdo} by considering the cases $ m = 1 $ and $ m > 1 $ separately.

\noindent\textbf{Case 1: $ m = 1 $} \\
In this case, the system \eqref{gdo} reduces to:
\begin{equation}\label{m1}
\dot{x} = y, \quad
\dot{y} = -\alpha y - (\epsilon + \sigma) x.
\end{equation}
The unique equilibrium point is the origin. The eigenvalues of the Jacobian matrix at the origin are given by
$\lambda_{1,2} = -\frac{1}{2}\alpha \pm \frac{1}{2} \sqrt{\alpha^2 - 4(\epsilon + \sigma)}.$

\begin{itemize}
    \item For $ \epsilon + \sigma < 0 $, the eigenvalues are real and of opposite signs, implying that the origin is a saddle point.
    \item For $ \epsilon + \sigma > 0 $:
    \begin{itemize}
        \item[$\triangleright$] When $ \alpha > 0 $ (resp. $ \alpha < 0 $), the eigenvalues have negative (resp. positive) real parts, making the origin stable (resp. unstable). Furthermore,
        \begin{itemize}
            \item If $ \alpha^2 < 4(\epsilon + \sigma) $, the eigenvalues are complex conjugates, and the origin is a focus.
            \item If $ \alpha^2 > 4(\epsilon + \sigma) $, the eigenvalues are real and distinct, and the origin is a node.
        \end{itemize}
      \item[$\triangleright$] When $ \alpha = 0 $, the eigenvalues are purely imaginary ($\lambda_{1,2} = \pm i \sqrt{\epsilon + \sigma}$), and the origin is a center.
    \end{itemize}
\end{itemize}

\noindent\textbf{Case 2: $ m > 1 $} \\
Here, the system \eqref{gdo} has up to three equilibria. Setting \( \dot{x} = 0 \) and \( \dot{y} = 0 \), we obtain the equation $-\epsilon x^m - \sigma x = 0,$ which factors as $x (\epsilon x^{m-1} + \sigma) = 0.$Thus, the equilibrium points are given by \eqref{eqneeed}.
    The existence of $ E^{\pm} $ depends on the values of $ \sigma $ and $ \epsilon $:
\begin{itemize}
    \item If $ m $ is odd, $ E^{\pm} $ exist only when $ \sigma \cdot \epsilon < 0 $.
    \item If $ m $ is even, $E^{+}$ always exists.
\end{itemize}

The Jacobian matrix at an equilibrium $(x^*, 0)$ is
\begin{equation*}
J(x^*, 0) = \begin{pmatrix}
0 & 1 \\
-\epsilon m (x^*)^{m-1} - \sigma & -\alpha
\end{pmatrix}.
\end{equation*}
\begin{itemize}
    \item At the origin $(0, 0)$, the eigenvalues are
    \(\lambda_{1,2} = -\frac{1}{2}\alpha \pm \frac{1}{2} \sqrt{\alpha^2 - 4\sigma}.\)
    \item At $ E^+ = \left( \sqrt[m-1]{-\frac{\sigma}{\epsilon}}, 0 \right) $, the eigenvalues are
    \(\lambda^{+}_{1,2} = -\frac{1}{2}\alpha \pm \frac{1}{2} \sqrt{\alpha^2 + 4\sigma(m-1)}.\)
    \item At $ E^- = \left( -\sqrt[m-1]{-\frac{\sigma}{\epsilon}}, 0 \right) $, the eigenvalues are
    \(\lambda^{-}_{1,2} = -\frac{1}{2}\alpha \pm \frac{1}{2} 
    \sqrt{\alpha^2 - 4\sigma \left((-1)^m  m +1\right)}.\)
\end{itemize}
The stability of these equilibria is determined by the signs of the real parts of the eigenvalues, as summarized in Tables \eqref{tab1} and \eqref{tab2} for odd and even $ m $, respectively. The proof is straightforward and supported by Figures \ref{M11}, \ref{M22} and \ref{M33}.

\end{proof}

The case \(\alpha = 0\) in the generalized Duffing system for \(m > 1\) presents a particular challenge due to the nonlinear nature of the system. When \(\alpha = 0\), the damping term vanishes, potentially leading to more complex and richer dynamical behaviors, such as the formation of closed orbits. In this context, the equilibria of the system can exhibit purely imaginary eigenvalues, suggesting the possibility of Hopf bifurcations or the formation of centers. However, unlike the linear case (\(m = 1\)), the nonlinearity of the system for \(m > 1\) makes the stability analysis depend critically on the parity of \(m\) and the signs of the parameters \(\sigma\) and \(\epsilon\). As a consequence of Proposition \ref{prop_sing}, we obtain the following results, which provides precise conditions under which the equilibria \((0, 0)\), \(E^+\), and \(E^-\) exhibit purely imaginary eigenvalues, thereby setting the stage for a more detailed study of bifurcations and periodic behaviors in the system.

\begin{figure}[t!]
\centering
\includegraphics[width=.22\columnwidth,height=.2\columnwidth]{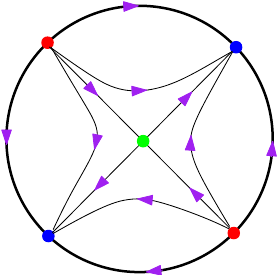}%
\hfill
\includegraphics[width=.22\columnwidth,height=.2\columnwidth]{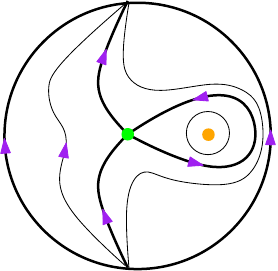}%
\hfill
\includegraphics[width=.22\columnwidth,height=.2\columnwidth]{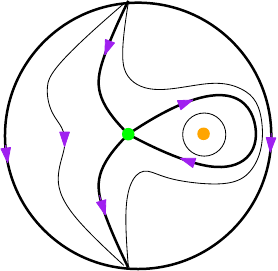}%
\hfill
\includegraphics[width=.22\columnwidth,height=.2\columnwidth]{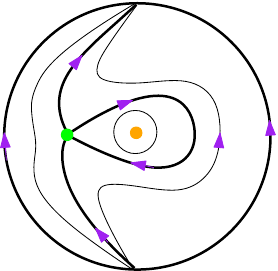}

\smallskip
\begin{minipage}{0.22\columnwidth}
\centering\tiny (a) $m=1$, \text{and} $\alpha=0$, $\epsilon+\sigma<0$
\end{minipage}%
\hfill
\begin{minipage}{0.22\columnwidth}
\centering\tiny(b) $m=2$, \text{and} $\alpha=0$, $\epsilon<0, \sigma<0$
\end{minipage}%
\hfill
\begin{minipage}{0.22\columnwidth}
\centering\tiny (c) $m=2$, \text{and} $\alpha=0$, $\epsilon>0, \sigma<0$
\end{minipage}%
\hfill
\begin{minipage}{0.22\columnwidth}
\centering\tiny (d) $m=2$, \text{and} $\alpha=0$, $\epsilon<0, \sigma>0$
\end{minipage}

\medskip
\includegraphics[width=.22\columnwidth,height=.2\columnwidth]{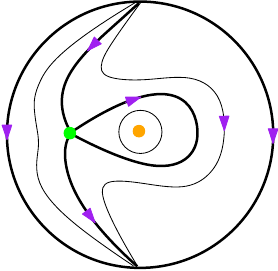}%
\hfill
\includegraphics[width=.22\columnwidth,height=.2\columnwidth]{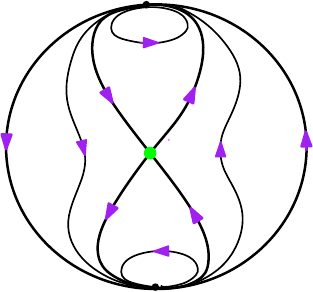}%
\hfill
\includegraphics[width=.22\columnwidth,height=.2\columnwidth]{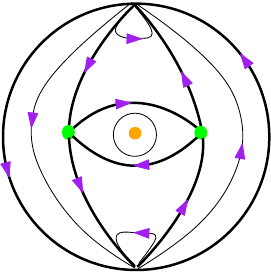}%
\hfill
\includegraphics[width=.22\columnwidth,height=.2\columnwidth]{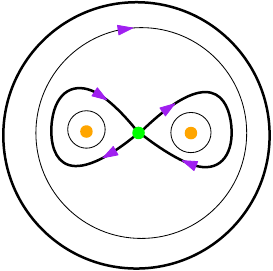}

\smallskip
\begin{minipage}{0.22\columnwidth}
\centering\tiny (e) $m=2$, \text{and} $\alpha=0$, $\epsilon>0, \sigma>0$
\end{minipage}%
\hfill
\begin{minipage}{0.22\columnwidth}
\centering\tiny (f) $m=3,5$, \text{and} $\alpha=0$, $\epsilon, \sigma<0$
\end{minipage}%
\hfill
\begin{minipage}{0.22\columnwidth}
\centering\tiny (g) $m=3,5$, \text{and} $\alpha=0$, $\epsilon<0, \sigma>0$
\end{minipage}%
\hfill
\begin{minipage}{0.22\columnwidth}
\centering\tiny (h) $m=3,5$, \text{and} $\alpha=0$, $\sigma<0, \epsilon>0$
\end{minipage}

\caption{\label{M44} The global phase portraits of system \eqref{gdo} for $\alpha=0$.}  
\end{figure}
 
\begin{cor}\label{cor_pair}
Suppose that $m>1$. Then the finite equilibria $(0, 0)$, $E^+$, and $E^-$ exhibit a pair of purely imaginary eigenvalues under the following conditions:
\begin{itemize}
    \item[(i)] For $(0, 0)$, the condition is $\alpha = 0$ and $\sigma > 0$.

\item[(ii)] For $E^+$, the condition is $\alpha = 0$ and $\sigma < 0$.

\item[(iii)] For $E^-$, the condition depends on the parity of $m$.
   \begin{itemize}
      \item If $m$ is odd, $\alpha = 0$ and $\sigma < 0$;
      \item If $m$ is even, $\alpha = 0$ and $\sigma > 0$.
   \end{itemize}
   \end{itemize}
\end{cor}
\begin{cor}\label{lem_center} 
Assume that $m>1$. The origin is the unique equilibrium point of the polynomial differential system \eqref{gdo} if, and only if, $m$ is odd and  $\sigma \epsilon > 0$.
\end{cor}

 In the scenario \(\alpha = 0\), the system exhibits different types of cycles depending on the value of \(m\) and the signs of the parameters \(\sigma\) and \(\epsilon\). These behaviors are characterized by the presence of homoclinic and heteroclinic cycles, as well as the configuration of equilibrium points, which can either correspond to saddles or centers, see Figure \ref{fig_new}. The nature of these cycles and equilibrium points is crucial for understanding the long-term dynamics of the system, as we will describe in the following proposition.
\begin{figure}[t!]
\centering
\includegraphics[width=.22\columnwidth,height=.22\columnwidth]{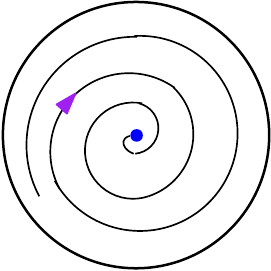}%
\hfill
\includegraphics[width=.22\columnwidth,height=.22\columnwidth]{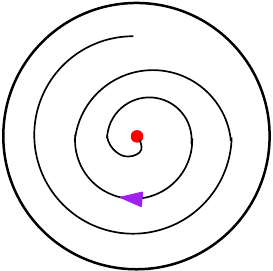}%
\hfill
\includegraphics[width=.22\columnwidth,height=.22\columnwidth]{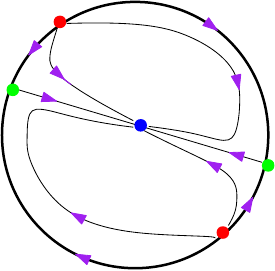}%
\hfill
\includegraphics[width=.23\columnwidth,height=.22\columnwidth]{m135}

\smallskip
\begin{minipage}{0.22\columnwidth}
\centering\tiny (a) $\alpha^2<4(\epsilon+\sigma), \alpha>0$
\end{minipage}%
\hfill
\begin{minipage}{0.22\columnwidth}
\centering\tiny (b) $\alpha^2<4(\epsilon+\sigma), \alpha<0$
\end{minipage}%
\hfill
\begin{minipage}{0.22\columnwidth}
\centering\tiny (c) $\alpha^2>4(\epsilon+\sigma),$ $\alpha>0,$ $ \epsilon+\sigma>0$
\end{minipage}%
\hfill
\begin{minipage}{0.24\columnwidth}
\centering\tiny (d) $\alpha^2>4(\epsilon+\sigma), \epsilon+\sigma<0$
\end{minipage}

\medskip
\centering
\includegraphics[width=.22\columnwidth,height=.2\columnwidth]{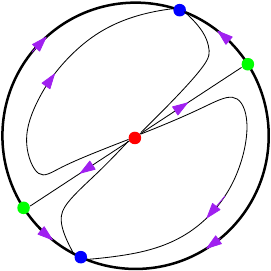}%
\hspace{0.1\columnwidth}
\includegraphics[width=.22\columnwidth,height=.22\columnwidth]{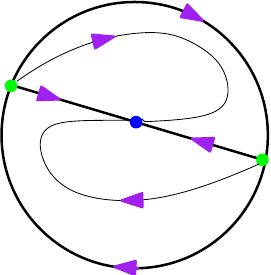}%
\hspace{0.1\columnwidth}
\includegraphics[width=.22\columnwidth,height=.22\columnwidth]{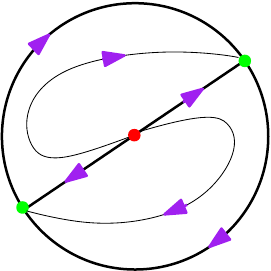}

\smallskip
\centering
\begin{minipage}{0.22\columnwidth}
\centering\tiny (e) $\alpha^2>4(\epsilon+\sigma),$ $\alpha<0,$ $ \epsilon+\sigma>0$
\end{minipage}%
\hspace{0.1\columnwidth}
\begin{minipage}{0.22\columnwidth}
\centering\tiny (f) $\alpha^2=4(\epsilon+\sigma), \alpha>0$
\end{minipage}%
\hspace{0.1\columnwidth}
\begin{minipage}{0.22\columnwidth}
\centering\tiny (g) $\alpha^2=4(\epsilon+\sigma), \alpha<0$
\end{minipage}

\caption{\label{M11} The global phase portraits of system \eqref{gdo} for $m=1$.}  
\end{figure}
\begin{prop}\label{new_prop}
Consider system \eqref{gdo} with $m > 1$ and assume $I_0 = \{\alpha = 0\}$. The following dynamical behaviors occur on $I_0$:
\begin{itemize}
\item[(i)] For even $m$, system \eqref{gdo} exhibits a homoclinic cycle.
\item[(ii)] For odd $m$ with $\sigma > 0$ and $\epsilon < 0$, system \eqref{gdo} undergoes a heteroclinic cycle.
\item[(iii)] For odd $m$ with $\sigma < 0$ and $\epsilon > 0$, system \eqref{gdo} admits a double-homoclinic cycle.
\end{itemize}
\end{prop}
\begin{proof}
When $\alpha = 0$, system \eqref{gdo} reduces to the Newtonian system
\begin{equation}\label{gdoNewton}
    \dot{x} = y, \quad \dot{y} = -\epsilon x^m - \sigma x = -\dot{U}(x),
\end{equation}
with potential energy
\begin{equation*}
U(x) = \dfrac{\epsilon}{m+1}x^{m+1} + \dfrac{\sigma}{2}x^2.
\end{equation*}
The critical points of $U(x)$ correspond to equilibrium points of \eqref{gdoNewton}, and their stability is determined by Theorem~\eqref{maxmin}.

In the case where $m$ is even, depending on the sign of $\sigma$, the potential $U(x)$ has exactly one strict local maximum and one strict local minimum, located at $x = 0$ and $x = \left(-\sigma/\epsilon\right)^{1/(m-1)}$. By Theorem~\eqref{maxmin}, the maximum corresponds to a saddle point and the minimum to a center. Therefore, there exists a homoclinic cycle, which is a trajectory connecting the saddle point to itself and enclosing the center.

For $m$ odd, two subcases arise:

\begin{itemize}
    \item Subcase 1: $\sigma > 0$ and $\epsilon < 0$: In this case, $U(x)$ has two strict local maxima at $x = \pm (-\sigma/\epsilon)^{1/(m-1)}$ and a strict local minimum at $x = 0$. By Theorem~\eqref{maxmin}, the maxima correspond to saddle points, while $(0,0)$ is a center. Hence, a heteroclinic cycle exists, consisting of trajectories connecting the two saddle points.
    
    \item Subcase 2: $\sigma < 0$ and $\epsilon > 0$: Here, $U(x)$ has a strict local maximum at $x = 0$ and two strict local minima at $x = \pm (-\sigma/\epsilon)^{1/(m-1)}$. By Theorem~\eqref{maxmin}, $(0,0)$ is a saddle point, while the minima correspond to centers. Therefore, a double-homoclinic cycle exists, consisting of two separate trajectories each connecting the saddle point to itself, forming a figure-eight pattern around both centers.
\end{itemize}

\end{proof}






\begin{figure}[t!]
\centering
\begin{minipage}{0.45\columnwidth}
\centering
\includegraphics[width=\linewidth,keepaspectratio]{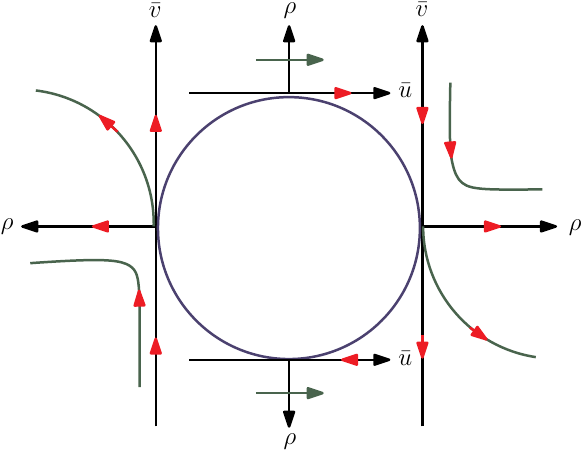}\\
\vspace{4pt}
\tiny (a) $\epsilon>0$
\end{minipage}%
\hspace{0.05\columnwidth}%
\begin{minipage}{0.45\columnwidth}
\centering
\includegraphics[width=\linewidth,keepaspectratio]{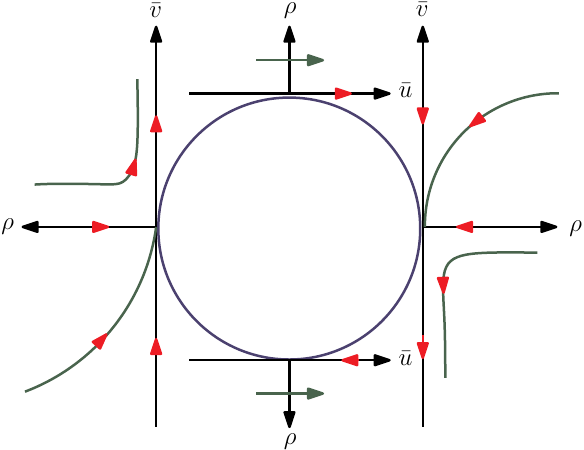}\\
\vspace{4pt}
\tiny (b) $\epsilon<0$
\end{minipage}

\caption{\label{LPo1e} The qualitative properties of vector fields $X_1^{+}$, $X_1^{-}$, $Y_1^{+}$, and $Y_1^{-}$.}
\end{figure}
 
We conclude this section by providing a very interesting result on the non-existence of limit cycles in system \eqref{gdo}$|_{\alpha\neq 0}$.
\begin{lem}\label{L5}
System \eqref{gdo}$|_{\alpha\neq 0}$ has no limit cycles.
\end{lem}

\begin{proof}
Assume that $F$ is the vector field associated to system \eqref{gdo}$|_{\alpha\neq 0}$. Emphasize that the divergence of $F$ is $-\alpha\neq 0$, thus from \textit{Bendison's criterion} (see section \ref{sec:ben_crit}) the result follows.
\end{proof}

\subsection{Global Dynamics}\label{GB}
In this section, we investigate the global dynamics of the Generalized Duffing Oscillator system. Specifically, we analyze the behavior of the system for different values of the parameter \(m\). The results will focus on the nature of the equilibria, their stability, and the corresponding phase portraits under varying conditions. The following propositions describe the key findings for any arbitrary number \(m.\)

\begin{prop}\label{thm3}
For \(m=1\), the differential system \eqref{gdo} has no infinite equilibrium points 
at the origin in the local chart \(U_2,\) and has at most two infinite equilibrium points \(P^{\pm}=\,\left(-\frac{1}{2}\alpha\pm \frac{1}{2} \sqrt{\alpha^2-4 (\epsilon+\sigma)}, 0\right)\) in the local charts \(U_1.\) 
\begin{itemize}
\item[(a)]
If \(\alpha^2<4(\epsilon+\sigma),\) there is no infinite equilibrium points in the local chart \(U_1.\)
\item[(b)]
If \(\alpha^2>4(\epsilon+\sigma),\) the equilibria \(P^{\pm}\) are a saddle (resp. unstable node) equilibria when $\epsilon+\sigma>0$ and  $\alpha>0$.
\item[(c)]
If \(\alpha^2>4(\epsilon+\sigma),\) \(P^{\pm}\) are a stable node (resp. saddle) for $\epsilon+\sigma>0$ and  $\alpha<0.$
\item[(d)] For \(\epsilon+\sigma<0,\) equilibria \(P^{\pm}\) are stable (resp. unstable) node. 
\item[(e)] For \(\alpha^2-4(\epsilon+\sigma)=0,\) there is only the equilibrium \(\left(-\frac{1}{2}\alpha, 0\right)\) in the local chart \(U_1\) which is a saddle-node point.
\end{itemize}
\end{prop}

\begin{proof}
System \eqref{m1} on the local chart \(U_2\) becomes 
\begin{equation*}
\dot{u}=1+\alpha u+(\epsilon+\sigma) u^2, \quad \dot{v}=\alpha v+(\epsilon+\sigma) uv,
\end{equation*}
and it has no infinite equilibrium points at the origin in this chart.

Now, the expression of system \eqref{m1} in the local chart \(U_1\) is 
\begin{equation*}
\dot{u}=-\epsilon-\sigma-\alpha u-u^2, \quad \dot{v}=- uv,
\end{equation*}
and has two infinite equilibrium points \(P^{\pm}\) for \(\alpha^2>4(\epsilon+\sigma).\) Since the associated Jacobian matrix is given by
\begin{align*}
\left( \begin {array}{cc} 
\mp\sqrt {{\alpha}^{2}-4(\epsilon+\sigma)}&0\\
0&\frac{1}{2}\alpha\mp\frac{1}{2} \sqrt {{\alpha}^{2}-4(\epsilon+\sigma)}
\end {array} \right),
\end{align*}
then statements (b)-(d) will be confirmed. When \(\alpha^2-4(\epsilon+\sigma)=0,\) there is only the equilibrium \(\left(-\frac{1}{2}\alpha, 0\right)\) which is a semi-hyperbolic point. By using \cite[Theorem 2.19]{MR2256001} it is leads to conclude that it is a saddle-node.
\end{proof}

\begin{figure}[t!]
\centering
\begin{minipage}{0.48\columnwidth} 
\centering
\includegraphics[width=\linewidth,height=0.75\columnwidth,keepaspectratio]{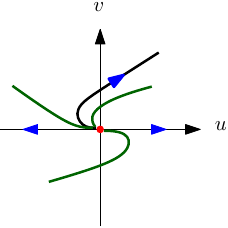}\\
\vspace{4pt}
\tiny (a) $\epsilon>0$
\end{minipage}%
\hspace{0.012\columnwidth}
\begin{minipage}{0.48\columnwidth}
\centering
\includegraphics[width=\linewidth,height=0.75\columnwidth,keepaspectratio]{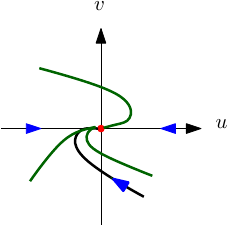}\\
\vspace{4pt}
\tiny (b) $\epsilon<0$
\end{minipage}

\caption{\label{LPoe} The local phase portraits of the origin of system \eqref{newnew1e}.}   
\end{figure}
\begin{rem}\label{TCM1}
    For the differential system \eqref{gdo} with \( m = 1 \), the origin \((0, 0)\) is the unique equilibrium point (both finite and at infinity) if, and only if, \(\alpha^2 < 4(\epsilon + \sigma)\). Combining this with Proposition \ref{prop_sing}, we conclude that \(\epsilon + \sigma > 0\) and \(\alpha = 0\)  are necessary and sufficient conditions for the origin to be a global center.
\end{rem}

\begin{prop}\label{thm1}
For $m $ even, the differential system \eqref{gdo} has no infinite equilibrium points in the local chart $U_1$, while it possesses a nilpotent equilibrium at $P=(0, 0)$ in the local chart $U_2$.
\begin{itemize}
\item[(a)]
For \(\epsilon>0\), $P$ is an unstable node with three parabolic sectors.
\item[(b)]
For \(\epsilon<0,\) $P$ is a stable node with three parabolic sectors.
\end{itemize}
\end{prop}
\begin{proof}
Suppose \( m = 2n  \) for some integer \( n \geq 1 \).   System \eqref{gdo} on the local chart \(U_1\) becomes 
\begin{align*}
\dot{u}= -\epsilon-v^{2n-1}(\sigma+\alpha u+u^2),\qquad
\dot{v}= -v^{2n}u,
\end{align*}
and it has no infinite singular points on \(v=0\) in this chart.
On the local chart \(U_2\) system \eqref{gdo} becomes
\begin{align}\label{newnew1e}
\dot{u}=\epsilon u^{2n+1}+v^{2n-1}(1+\alpha u+\sigma u^2),\qquad
\dot{v}=\epsilon u^{2n}v+v^{2n}(\alpha+\sigma u).
\end{align}
The origin of \(U_2,\) O, is an equilibrium point.  To investigate the higher-order equilibrium \((0,0),\) we employ a quasi-homogeneous blow-up technique. Based on the Newton diagram of system \eqref{newnew1e}, the appropriate blow-up is chosen using coefficient \( (2n, 2n-1 ), \) \ie
\begin{equation}\label{qu1e}
(u, v) =(\bar{u} \rho^{2n}, \bar{v} \rho^{2n-1}). 
\end{equation}
The quasi-homogeneous blow-up transformation \eqref{qu1e} with $\bar{u} = 1$ (resp. $\bar{u} = -1, \bar{v} = 1, \bar{v} = -1$), converts system \eqref{newnew1} to the associated vector field $X_1^{+}$ (resp. $X_1^{-}, Y_1^{+}, Y_1^{-}$). These vector fields are defined as follows
\begin{align*}
X_1^{+}:
\begin{cases}
\dfrac{d\rho}{d\eta} = \dfrac {\epsilon{\rho}^{6n}+\rho{v}^{2n-1}
+\alpha{\rho}^{2n+1}{v}^{2n-1}+\sigma{\rho}^{4n+1}{v}^{2n-1}}{2n},\\\\\nonumber
\dfrac{d\bar{v}}{d\eta} = -{\dfrac {2n{v}^{2n}-\epsilon{\rho}^{6n-1}v-{v}^{2n}
-\alpha{\rho}^{2n}{v}^{2n}-\sigma{\rho}^{4n}{v}^{2n}}{2n}},
\end{cases}
\end{align*}

\begin{align*}
X_1^{-}:
\begin{cases}
\dfrac{d\rho}{d\eta} =\dfrac {{\epsilon \rho}^{6n}-\rho{v}^{2n-1}
+\alpha{\rho}^{2n+1}{v}^{2n-1}-\sigma{\rho}^{4n+1}{v}^{2n-1}}{2n}, \\\\\nonumber
\dfrac{d\bar{v}}{d\eta} =\dfrac {2n{v}^{2n}+\epsilon{\rho}^{6n-1}v-{v}^{2n}
+\alpha{\rho}^{2n}{v}^{2n}-\sigma{\rho}^{4n}{v}^{2n}}{2n},
\end{cases}
\end{align*}

\begin{align*}
Y_1^{+}:
\begin{cases}
\dfrac{d\rho}{d\eta} = \mathcal{O}(|\rho, \bar{v}|), \\\\
\dfrac{d\bar{u}}{d\eta} = 1 + \mathcal{O}(|\rho, \bar{v}|),
\end{cases}
\quad\quad
Y_1^{-}:
\begin{cases}
\dfrac{d\rho}{d\eta} = \mathcal{O}(|\rho, \bar{v}|), \\\\
\dfrac{d\bar{u}}{d\eta} = -1 + \mathcal{O}(|\rho, \bar{v}|),
\end{cases}
\end{align*}
and \(d\eta = \rho^{4n^2-6n+1} d\tau.\) 

 The vector fields \(X_1^{+}\) and  \(X_1^{-}\) increase along the \(\rho\)-axis when \(\epsilon > 0\), and decreases for \(\epsilon < 0\). Conversely, along the \(\bar{v}\)-axis, \(X_1^{+}\) (resp.  \(X_1^{-}\)) consistently exhibits a decreasing (resp. increasing) behavior.
Hence, the following statements describe these vector fields:
\begin{itemize}
\item The only equilibria of \(X_1^{+}\) and \(X_1^{-}\) on \(\rho=0\) is the origin, and it is a saddle node.
\item
For the vector fields \( Y_1^{+} \) and \( Y_1^{-} \), the origin is a regular point.
\end{itemize}
Therefore, Figure \ref{LPo1e} illustrates the qualitative properties of the vector fields $X_1^{+},$ $X_1^{-}, Y_1^{+},$ and $Y_1^{-}.$ By shrinking the circles in these figures to the point O of system \eqref{newnew1e}, we obtain Figure \ref{LPoe}.
 
\begin{figure}[t!]
\centering
\includegraphics[width=.22\columnwidth,height=.22\columnwidth]{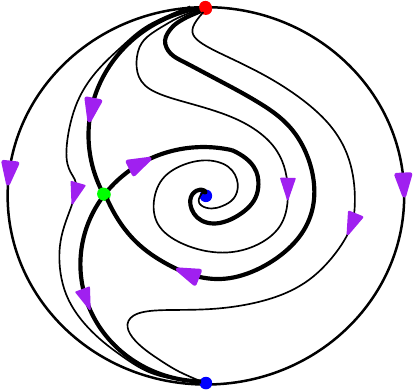}%
\hfill
\includegraphics[width=.22\columnwidth,height=.22\columnwidth]{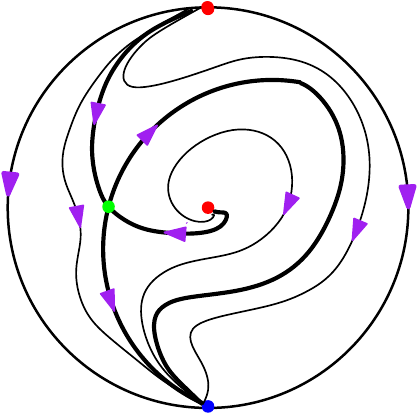}%
\hfill
\includegraphics[width=.22\columnwidth,height=.22\columnwidth]{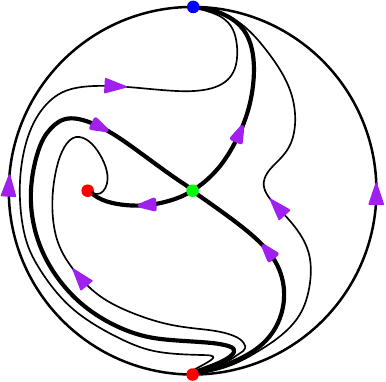}%
\hfill
\includegraphics[width=.22\columnwidth,height=.22\columnwidth]{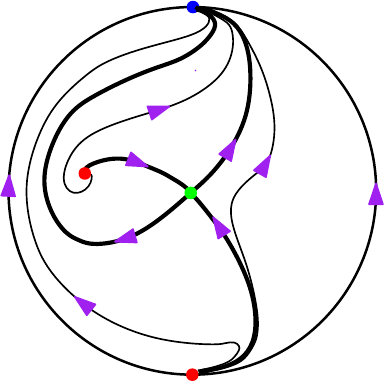}

\smallskip
\begin{minipage}{0.22\columnwidth}
\centering\tiny (a) $\sigma>0$, $\epsilon>0$, $\alpha>0$
\end{minipage}%
\hfill
\begin{minipage}{0.22\columnwidth}
\centering\tiny (b) $\sigma>0$, $\epsilon>0$, $\alpha<0$
\end{minipage}%
\hfill
\begin{minipage}{0.22\columnwidth}
\centering\tiny (c) $\sigma<0$, $\epsilon<0$, $\alpha>0$
\end{minipage}%
\hfill
\begin{minipage}{0.22\columnwidth}
\centering\tiny (d) $\sigma<0$, $\epsilon<0$, $\alpha<0$
\end{minipage}

\medskip
\includegraphics[width=.22\columnwidth,height=.22\columnwidth]{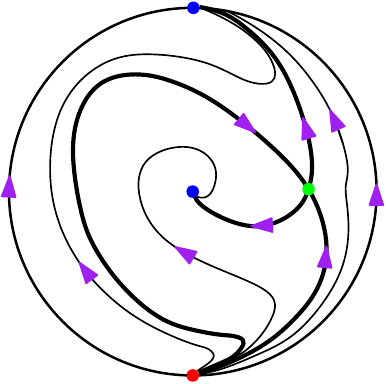}%
\hfill
\includegraphics[width=.22\columnwidth,height=.22\columnwidth]{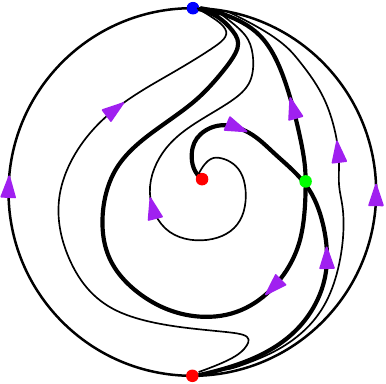}%
\hfill
\includegraphics[width=.22\columnwidth,height=.22\columnwidth]{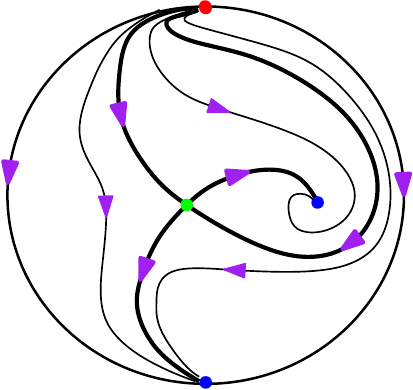}%
\hfill
\includegraphics[width=.22\columnwidth,height=.22\columnwidth]{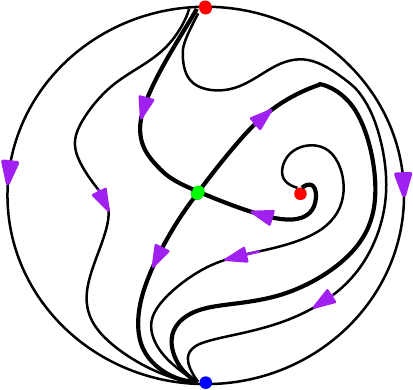}

\smallskip
\begin{minipage}{0.22\columnwidth}
\centering\tiny (e) $\sigma>0$, $\epsilon<0$, $\alpha>0$
\end{minipage}%
\hfill
\begin{minipage}{0.22\columnwidth}
\centering\tiny (f) $\sigma>0$, $\epsilon<0$, $\alpha<0$
\end{minipage}%
\hfill
\begin{minipage}{0.22\columnwidth}
\centering\tiny (g) $\sigma<0$, $\epsilon>0$, $\alpha>0$
\end{minipage}%
\hfill
\begin{minipage}{0.22\columnwidth}
\centering\tiny (h) $\sigma<0$, $\epsilon>0$, $\alpha<0$
\end{minipage}

\caption{\label{M22} The global phase portraits of system \eqref{gdo} for $m$ even.}
\end{figure}
 
\end{proof}

\begin{prop}\label{thm2}
When \(m>1\) is odd, the differential system \eqref{gdo} has no infinite equilibrium points in the local chart \(U_1,\) and has the linearly zero equilibrium point \(P= (0, 0)\) in the local chart \(U_2.\)
\begin{itemize}
\item[(a)]
The local phase portrait of system \eqref{gdo} around \(P\) is formed by two hyperbolic sectors for \(\epsilon>0.\)
\item[(b)]
The local phase portrait of system \eqref{gdo} around \(P\) is  formed by two parabolic and two elliptic sectors for \(\epsilon<0\) and \(\sigma<0.\)
\item[(c)]
The local phase portrait of system \eqref{gdo} around \(P\) is  formed by  four parabolic and two elliptic sectors when \(\epsilon<0\) and \(\sigma>0.\)
\end{itemize}
\end{prop}

\begin{proof}
Suppose \( m = 2n + 1 \) for some integer \( n \geq 1 \).   System \eqref{gdo} on the local chart \(U_1\) becomes 
\begin{align*}
\dot{u}= -\epsilon-v^{2n}(\sigma+\alpha u+u^2),\qquad
\dot{v}= -v^{2n+1}u,
\end{align*}
and it has no infinite singular points on \(v=0\) in this chart.
On the local chart \(U_2\) system \eqref{gdo} becomes
\begin{align}\label{newnew1}
\dot{u}=\epsilon u^{2n+2}+v^{2n}(1+\alpha u+\sigma u^2),\qquad
\dot{v}=\epsilon u^{2n+1}v+v^{2n+1}(\alpha+\sigma u).
\end{align}
 \begin{figure}[t!]
\centering
\begin{minipage}{0.45\columnwidth}
\centering
\includegraphics[width=\linewidth,keepaspectratio]{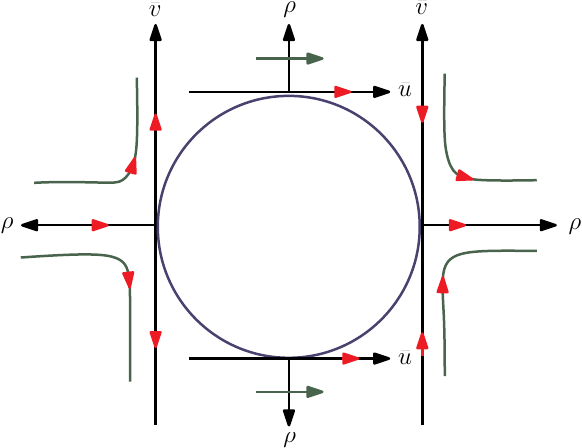}\\
\vspace{4pt}
\tiny (a) $\epsilon>0$
\end{minipage}%
\hspace{0.05\columnwidth}%
\begin{minipage}{0.45\columnwidth}
\centering
\includegraphics[width=\linewidth,keepaspectratio]{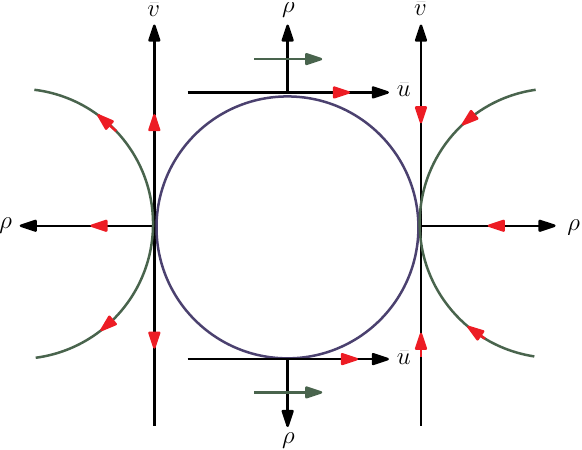}\\
\vspace{4pt}
\tiny (b) $\epsilon<0$
\end{minipage}

 \caption{\label{LPo1} The qualitative properties of vector fields $X_2^{+},$ $X_2^{-}, Y_2^{+},$ and $Y_2^{-}$.}  
\end{figure}
The origin of \(U_2,\) O, is an equilibrium point.  To investigate the higher-order equilibrium \((0,0),\) we employ a quasi-homogeneous blow-up technique. Based on the Newton diagram of system \eqref{newnew1}, the appropriate blow-up is chosen using coefficient \( (n+1, n ), \) \ie
\begin{equation}\label{qu1}
(u, v) =(\bar{u} \rho^{n+1}, \bar{v} \rho^{n}). 
\end{equation}
The quasi-homogeneous blow-up transformation \eqref{qu1} with $\bar{u} = 1$ (resp. $\bar{u} = -1, \bar{v} = 1, \bar{v} = -1$), maps system \eqref{newnew1} to the corresponding vector field $X_2^{+}$ (resp. $X_2^{-}, Y_2^{+}, Y_2^{-}$). These vector fields are given by
\begin{align*}
X_2^{+}:
\begin{cases}
\dfrac{d\rho}{d\eta} = \dfrac{\epsilon{\rho}^{4n+3}+\rho{\bar{v}}^{2n}+\alpha{\rho}^{n+2}{\bar{v}}
^{2n}+\sigma{\rho}^{2n+3}{\bar{v}}^{2n}}{n+1},  \\\\
\dfrac{d\bar{v}}{d\eta} = \dfrac{\epsilon{\rho}^{4n+2}\bar{v}+\alpha{\rho}^{n+1}{\bar{v}}^{2n+1}+\sigma{\rho}^{2n+2}{\bar{v}}^{2n+1}-n{\bar{v}}^{2n+1}}{n+1},
\end{cases}
\end{align*}

\begin{align*}
X_2^{-}:
\begin{cases}
\dfrac{d\rho}{d\eta} = -\dfrac{\epsilon{\rho}^{4n+3}+\rho{\bar{v}}^{2n}-\alpha{\rho}^{n+2}{\bar{v}}
^{2n}+\sigma{\rho}^{2n+3}{\bar{v}}^{2n}}{n+1},  \\\\
\dfrac{d\bar{v}}{d\eta} = \dfrac{-\epsilon{\rho}^{4n+2}\bar{v}+\alpha{\rho}^{n+1}{\bar{v}}^{2n+1}-\sigma{\rho}^{2n+2}{\bar{v}}^{2n+1}+n{\bar{v}}^{2n+1}}{n+1},
\end{cases}
\end{align*}

\begin{align*}
Y_2^{+}:
\begin{cases}
\dfrac{d\rho}{d\eta} = \mathcal{O}(|\rho, \bar{v}|), \\\\
\dfrac{d\bar{u}}{d\eta} = 1 + \mathcal{O}(|\rho, \bar{v}|),
\end{cases}
\quad\quad
Y_2^{-}:
\begin{cases}
\dfrac{d\rho}{d\eta} = \mathcal{O}(|\rho, \bar{v}|), \\\\
\dfrac{d\bar{u}}{d\eta} = 1+ \mathcal{O}(|\rho, \bar{v}|),
\end{cases}
\end{align*}
and \(d\eta = \rho^{2n^2-n-1} d\tau.\) 

 The vector fields \(X_2^{+}\) and  \(X_2^{-}\) increase along the \(\rho\)-axis for \(\epsilon > 0\), and decreases when \(\epsilon < 0\). On the other hand, along the \(\bar{v}\)-axis, the trajectories of vector field \(X_2^{+}\) (resp.  \(X_2^{-}\)) are attracted to (repelled from) the origin.
Therefore, the following statements describe these vector fields:
\begin{itemize}
\item The only equilibria of \(X_2^{+}\) on \(\rho=0\) is the origin, and it is a saddle (resp. stable node) for \(\epsilon>0\) (resp. \(\epsilon<0\)).
\item The only equilibria of \(X_2^{-}\) on \(\rho=0\) is the origin while it is a saddle (resp. unstable node) for \(\epsilon>0\) (resp. \(\epsilon<0\)).
\item
For the vector fields \( Y_2^{+} \) and \( Y_2^{-} \), the origin is a regular point.
\end{itemize}
Therefore, Figure \ref{LPo1} illustrates the qualitative properties of the vector fields $X_2^{+},$ $X_2^{-}, Y_2^{+},$ and $Y_2^{-}.$ By shrinking the circles in these figures to the point O of system \eqref{newnew1}, we obtain Figure \ref{LPo}.
 
\begin{figure}[t!]
\centering
\includegraphics[width=.34\columnwidth,height=.35\columnwidth]{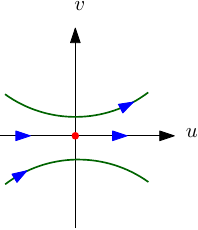}%
\hspace{0.5cm} 
\includegraphics[width=.34\columnwidth,height=.35\columnwidth]{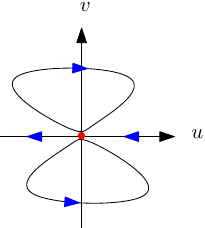}

\vspace{0.2cm} 
\begin{minipage}{0.34\columnwidth}
\centering (a) $\epsilon>0$
\end{minipage}%
\hspace{0.5cm} 
\begin{minipage}{0.34\columnwidth}
\centering (b) $\epsilon<0$
\end{minipage}

\caption{\label{LPo}The local phase portraits of the origin of system \eqref{newnew1}.}
\end{figure}
 
\end{proof}

\begin{proof}[Proof of  Theorem~\ref{1}]
    After analyzing the local dynamics of finite equilibrium points (Section~\ref{sec1}) and the global dynamics (Section~\ref{GB}), we conclude the proof of Theorem~\ref{1} by combining these results.
\end{proof}

\begin{figure}[H]
\centering
\includegraphics[width=.22\columnwidth,height=.22\columnwidth]{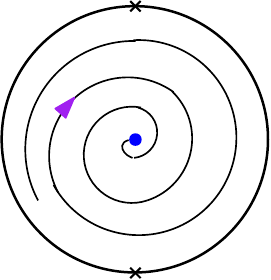}%
\hspace{0.3cm}%
\includegraphics[width=.22\columnwidth,height=.22\columnwidth]{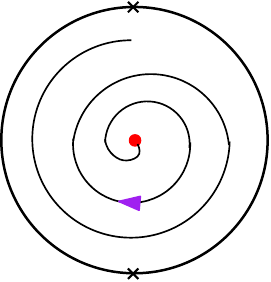}%
\hspace{0.3cm}%
\includegraphics[width=.22\columnwidth,height=.22\columnwidth]{m3nnp}%
\hspace{0.3cm}%
\includegraphics[width=.22\columnwidth,height=.22\columnwidth]{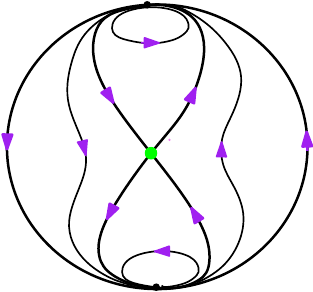}

\vspace{0.1cm}%
\begin{minipage}{0.22\columnwidth}
\centering\footnotesize (a) $\sigma>0$, $\epsilon>0$, $\alpha>0$
\end{minipage}%
\hspace{0.3cm}%
\begin{minipage}{0.22\columnwidth}
\centering\footnotesize (b) $\sigma>0$, $\epsilon>0$, $\alpha<0$
\end{minipage}%
\hspace{0.3cm}%
\begin{minipage}{0.22\columnwidth}
\centering\footnotesize (c) $\sigma<0$, $\epsilon<0$, $\alpha>0$
\end{minipage}%
\hspace{0.3cm}%
\begin{minipage}{0.22\columnwidth}
\centering\footnotesize (d) $\sigma<0$, $\epsilon<0$, $\alpha<0$
\end{minipage}

\vspace{0.2cm}%
\includegraphics[width=.22\columnwidth,height=.22\columnwidth]{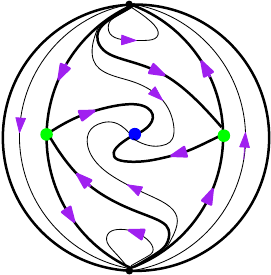}%
\hspace{0.3cm}%
\includegraphics[width=.22\columnwidth,height=.22\columnwidth]{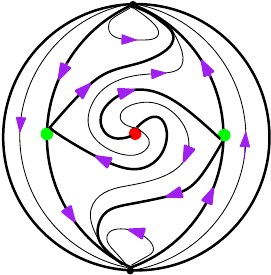}%
\hspace{0.3cm}%
\includegraphics[width=.22\columnwidth,height=.22\columnwidth]{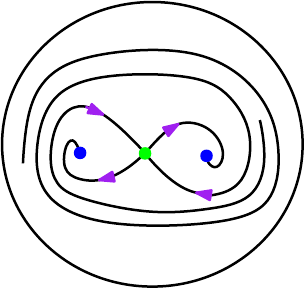}%
\hspace{0.3cm}%
\includegraphics[width=.22\columnwidth,height=.22\columnwidth]{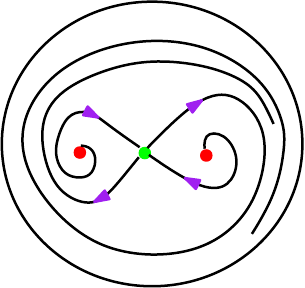}

\vspace{0.1cm}%
\begin{minipage}{0.22\columnwidth}
\centering\footnotesize (e) $\sigma>0$, $\epsilon<0$, $\alpha>0$
\end{minipage}%
\hspace{0.3cm}%
\begin{minipage}{0.22\columnwidth}
\centering\footnotesize (f) $\sigma>0$, $\epsilon<0$, $\alpha<0$
\end{minipage}%
\hspace{0.3cm}%
\begin{minipage}{0.22\columnwidth}
\centering\footnotesize (g) $\sigma<0$, $\epsilon>0$, $\alpha>0$
\end{minipage}%
\hspace{0.3cm}%
\begin{minipage}{0.22\columnwidth}
\centering\footnotesize (h) $\sigma<0$, $\epsilon>0$, $\alpha<0$
\end{minipage}

\caption{\label{M33}The global phase portraits of system \eqref{gdo} when $m$ is odd.}
\end{figure}

\section{Proof of Theorems \ref{prop1} and \ref{teoB}}\label{main_BC}

This section is dedicated to proving Theorems \ref{prop1} and \ref{teoB}. To do so, we must first prove the following result. 

\begin{prop}\label{prop_center} 
Suppose that $m>1$. The polynomial differential system \eqref{gdo} has a center at the origin if, and only if, $\alpha = 0$ and $\sigma > 0$.
\end{prop}

\begin{proof}
Corollary \ref{cor_pair} ensures that the eigenvalues of the Jacobian matrix associated with system \eqref{gdo} are purely imaginary whenever \(\alpha = 0\) and \(\sigma > 0\). 

By applying the linear transformation \((x, y) \mapsto (u, -\sqrt{\sigma} v)\), the generalized Duffing oscillator \eqref{gdo} for \(\alpha = 0\) takes the form of its Jordan canonical representation:
\begin{eqnarray*}\label{jordan0}
\dot{u} = -\sqrt{\sigma} v, \quad
\dot{v} = \sqrt{\sigma} u + \frac{\epsilon}{\sqrt{\sigma}} u^m.
\end{eqnarray*}
Next, by rescaling time as \( t = \frac{\tau}{\sqrt{\sigma}} \), we obtain the system
\begin{eqnarray}\label{eq_res}
\frac{du}{d\tau} = -v, \quad \frac{dv}{d\tau} = u + \frac{\epsilon}{\sigma} u^m.
\end{eqnarray}
We now define the function
\begin{equation*}
H(u, v) = u^2 + v^2 + \frac{2 \epsilon}{\sigma (m+1)} u^{m+1}.
\end{equation*}
We claim that \(H\) is a first integral of system \eqref{gdo}. Indeed, by computing the time derivative of \(H(u, v)\) along the trajectories of this system, we get
\begin{equation*}
\frac{dH}{d\tau} = \frac{\partial H}{\partial u} \frac{du}{d\tau} + \frac{\partial H}{\partial v} \frac{dv}{d\tau}.
\end{equation*}
Substituting \(\frac{du}{d\tau}\) and \(\frac{dv}{d\tau}\), we obtain
\begin{equation*}
\frac{dH}{d\tau} = \left(2u + \frac{2 \epsilon}{\sigma} u^m\right)(-v) + (2v)\left(u + \frac{\epsilon}{\sigma} u^m\right).
\end{equation*}
Thus,
\begin{equation*}
\frac{dH}{d\tau} = -2u v - \frac{2 \epsilon}{\sigma} u^m v + 2u v + \frac{2 \epsilon}{\sigma} u^m v = 0.
\end{equation*}
Since \( \frac{dH}{d\tau} = 0 \), \(H(u, v)\) is constant along trajectories of system \ref{eq_res}. Therefore, \(H\) is a first integral of this system, and the trajectories lie on the level curves of \(H\). This implies that system \ref{eq_res} has periodic orbits around the origin.

The result follows from the Poincaré-Lyapunov Theorem (see Theorem \ref{PL}), which ensures that the origin of system \ref{eq_res} (and therefore of system \eqref{gdo}) is a center.
\end{proof}
\begin{rem}
The center at the origin, guaranteed by Proposition \ref{prop_center}, is reversible. This is because the system \eqref{gdo} is invariant under the transformation $(u, v, t) \mapsto (-u, v, -t)$ when $m$ is odd. This symmetry implies that the trajectories of the system in the phase plane are symmetric with respect to the $v$-axis, which is a characteristic property of reversible centers.
\end{rem}
\begin{proof}[Proof of Theorem \ref{prop1}]
Based on Proposition \ref{prop_center} and Corollary \ref{lem_center}, the result follows.    
\end{proof}

Now, using the classification of the center of the system \eqref{gdo} provided in Theorem \ref{prop1}, we are ready to prove Theorem \ref{teoB}.

\begin{proof}[Proof of Theorem \ref{teoB}]
The case \( m = 1 \) follows directly from Remark~\ref{TCM1}, while the case \( m > 1 \) is a consequence of Theorem~\ref{prop1}, together with item~(a) of Proposition~\ref{thm2} and the conditions established in Proposition~\ref{prop_main}.
\end{proof}

\section{Acknowledgement}
Gabriel Rondón is partially supported by  the Agencia Estatal de Investigaci\'on of Spain grant PID2022-136613NB-100. Nasrin Sadri was in part supported by a grant from IPM (No. 1403370046), P. O. Box 19395-5746, Tehran, Iran.

\bibliographystyle{abbrv}
\bibliography{references1}

\begin{thebibliography}{10}

\bibitem{BC1}
I.~Bendixson.
\newblock Sur les courbes définies par des équations différentielles.
\newblock {\em Acta Math.}, 24:1--88, 1901.

\bibitem{Chandrasekar}
V.~Chandrasekar, S.~Pandey, M.~Senthilvelan, and M.~Lakshmanan.
\newblock A simple and unified approach to identify integrable nonlinear
  oscillators and systems.
\newblock {\em J. Math. Phys}, 47:023508, 2006.

\bibitem{Chandrasekar1}
V.~Chandrasekar, M.~Senthilvelan, and M.~Lakshmanan.
\newblock New aspects of integrability of force-free duffing-van der pol
  oscillator and related nonlinear systems.
\newblock {\em J. Phys. A}, 37:4527, 2004.

\bibitem{10.2307/2001320}
A.~Cima and J.~Llibre.
\newblock Bounded polynomial vector fields.
\newblock {\em Trans. Amer. Math. Soc.}, 318(2):557--579, 1990.

\bibitem{Demina}
M.~Demina.
\newblock Novel algebraic aspects of liouvillian integrability for
  two-dimensional polynomial dynamical systems.
\newblock {\em Phys. Lett. A}, 382(20):1353--1360, 2018.

\bibitem{Demina1}
M.~Demina.
\newblock Liouvillian integrability of the generalized duffing oscillators.
\newblock {\em Analysis and Mathematical Physics}, 11(25), 2021.

\bibitem{MR2256001}
F.~Dumortier, J.~Llibre, and J.~C. Art\'{e}s.
\newblock {\em Qualitative theory of planar differential systems}.
\newblock Universitext. Springer-Verlag, Berlin, 2006.

\bibitem{GalVill}
M.~Galeotti and M.~Villarini.
\newblock Some properties of planar polynomial systems of even degree.
\newblock {\em Ann. Mat. Pura Appl}, 161:299--313, 1992.

\bibitem{chen3}
C.~Hebai, Z.~Rui, and Z.~Xiang.
\newblock Dynamics of polynomial rayleigh-duffing system near infinity and its
  global phase portraits with a center.
\newblock {\em Advances in Mathematics}, 433:109326, 2023.

\bibitem{chen1}
C.~Hebai, F.~Zhaosheng, and Z.~Rui.
\newblock Nilpotent global centers of generalized polynomial kukles system with
  degree three.
\newblock {\em Proceedings of the American Mathematical Society},
  52(9):3785–3800, 2024.

\bibitem{chen2}
C.~Hebai, L.~Zhijie, and Z.~Rui.
\newblock A sufficient and necessary condition of generalized polynomial
  liénard systems with global centers.
\newblock {\em Proceedings of the American Mathematical Society}, 2022.

\bibitem{ref1}
Y.~Ilyashenko and S.~Yakovenko.
\newblock {\em Lectures on Analytic Differential Equations}, volume~86 of {\em
  Graduate Studies in Mathematics}.
\newblock American Mathematical Society, Providence, RI, 2008.

\bibitem{ref2}
A.~Liapunov.
\newblock {\em Problème Général de la Stabilité du Mouvement}.
\newblock Number~17 in Annals of Mathematics Studies. Princeton University
  Press, Princeton, N.J. and Oxford University Press, London, 1947.

\bibitem{LLIBRE202266}
J.~Llibre and C.~Valls.
\newblock Global centers of the generalized polynomial liénard differential
  systems.
\newblock {\em Journal of Differential Equations}, 330:66--80, 2022.

\bibitem{doi:10.1080/14689367.2023.2228737}
J.~Llibre and C.~Valls.
\newblock Reversible global centres with quintic homogeneous nonlinearities.
\newblock {\em Dynamical Systems}, 0(0):1--22, 2023.

\bibitem{Parthasarathy}
S.~Parthasarathy and M.~Lakshmanan.
\newblock On the exact solutions of the duffing oscillator.
\newblock {\em J. Sound Vib.}, 137:523--526, 1990.

\bibitem{perko}
L.~Perko.
\newblock {\em Differential Equations and Dynamical Systems}, volume~7 of {\em
  Texts in Applied Mathematics}.
\newblock Springer, 3rd edition, 2001.

\bibitem{ref3}
H.~Poincaré.
\newblock Mémoire sur les courbes définies par une équation différentielle.
\newblock {\em Journal de Mathématiques Pures et Appliquées}, 7:375--422,
  1881.
\newblock Also in Sér. 3, 8: 251--296, 1882; Sér. 4, 1: 167--244, 1882; Sér.
  4, 2: 151--217, 1886.

\bibitem{ref4}
V.~Romanovski and D.~Shafer.
\newblock {\em The Center and Cyclicity Problems: A Computational Algebra
  Approach}.
\newblock Birkhäuser Boston, Inc., Boston, MA, 2009.

\bibitem{Ruiz}
A.~Ruiz and C.~Muriel.
\newblock On the integrability of li\'enard i-type equations via
  \(\lambda\)-symmetries and solvable structures.
\newblock {\em Appl. Math. Comp.}, 339:888--898, 2018.

\bibitem{Salas}
A.~H. Salas and J.~E.~C. H.
\newblock Exact solution to duffing equation and the pendulum equation.
\newblock {\em Applied Mathematical Sciences}, 176:8781--8789, 2014.

\bibitem{Stachowiak}
T.~Stachowiak.
\newblock Hypergeometric first integrals of the duffing and van der pol
  oscillators.
\newblock {\em J. Differ. Equ}, 266:5895--5911, 2019.

\end{thebibliography}

\end{document}